\documentclass[12pt]{amsart}
\usepackage{amsmath,amssymb,enumerate}
\newtheorem{theorem}{Theorem}[section]

\newtheorem{lemma}[theorem]{Lemma}

\newtheorem{corollary}[theorem]{Corollary}
\newtheorem{conjecture}[theorem]{Conjecture}

\newcommand{\bR}{\mathbb R}

\newcommand{\cU}{\mathcal{U}}

\DeclareMathOperator{\si}{si}

\DeclareMathOperator{\cl}{cl}

\DeclareMathOperator{\PG}{PG}

\newcommand{\elem}{\epsilon}
\newcommand{\del}{\setminus}
\newcommand{\con}{/}

\begin{document}

\sloppy
\title[Matroids with no $U_{2,n}$-minor]{The number of points in a
matroid with no $n$-point line as a minor}

\author[Geelen]{Jim Geelen}

\author[Nelson]{Peter Nelson}
\address{Department of Combinatorics and Optimization,
University of Waterloo, Waterloo, Canada}
\thanks{ This research was partially supported by a grant from the
Natural Sciences and Engineering Research Council of Canada.}

\date{\today}

\begin{abstract}
For any positive integer $l$ we prove that if $M$ is 
a simple matroid with no $(l+2)$-point line as a minor and with
sufficiently large rank, then 
$|E(M)|\le \frac{q^{r(M)}-1}{q-1}$, where $q$ is the largest 
prime power less than or equal to $l$.
Equality is attained by projective geometries over GF$(q)$.
\end{abstract}

\maketitle

\section{Introduction}

Kung~[\ref{kung}] proved the following theorem.
\begin{theorem} \label{line}
For any integer $l\ge 2$, if $M$ is a simple matroid with 
no $U_{2,l+2}$-minor, then $|E(M)|\le \frac{l^{r(M)}-1}{l-1}$.
\end{theorem}
The above bound is tight in the case that $l$ is a prime power
and $M$ is a projective geometry. In fact, among matroids of rank at least
$4$, projective geometries are
the only matroids that attain the bound; see~[\ref{kung}].
Therefore, the bound is not tight when $l$ is not a prime power. 
We prove the following bound that was conjectured by 
Kung~[\ref{kung},\ref{gkw}].
\begin{theorem}\label{main}
Let $l\ge 2$ be a positive integer and let $q$ be the
largest prime power less than or equal to $l$.
If $M$ is a simple matroid with no $U_{2,l+2}$-minor
and with sufficiently large rank, then 
$|E(M)|\le \frac{q^{r(M)}-1}{q-1}$.
\end{theorem}
The case where $l = 6$ was resolved by Bonin and Kung in~[\ref{bk}].

We will also prove that the only matroids of large rank
that attain the bound in Theorem~\ref{main} are the
projective geometries over GF$(q)$; see Corollary~\ref{extreme}.

A matroid $M$ is {\em round} if $E(M)$ cannot be partitioned into
two sets of rank less than $r(M)$.
We prove Theorem~\ref{main} by reducing it to the following result.
\begin{theorem}\label{extra}
For each prime power $q$, there exists a positive integer~$n$ such that, if $M$ is a
round matroid with a PG$(n-1,q)$-minor but no $U_{2,q^2+1}$-minor, 
then $\epsilon(M) \le \frac{q^{r(M)}-1}{q-1}$.
\end{theorem}
For any integer $l\ge 2$,
there is an integer $k$ such that $2^{k-1}<l\le 2^k$. Therefore, 
if $q$ is the largest prime power less than or equal to $l$, then $l<2q$.
So, to prove Theorem~\ref{main}, it would suffice to prove the
weaker version of Theorem~\ref{extra} where $U_{2,q^2+1}$ is 
replaced by $U_{2,2q+1}$. With this in mind, we find the stronger
version somewhat surprising. 
 
We further reduce Theorem~\ref{extra} to the following result.
\begin{theorem}\label{longline}
For each prime power $q$ there exists an integer $n$ such that,
if $M$ is a round matroid that contains a $U_{2,q+2}$-restriction
and a PG$(n-1,q)$-minor, then $M$ contains a $U_{2,q^2+1}$-minor.
\end{theorem}

The following conjecture, if true, would imply all of the results above.
\begin{conjecture}\label{wild}
For each prime power $q$, there exists a positive integer~$n$ such that, if $M$ is a 
round matroid with a PG$(n-1,q)$-minor but no 
$U_{2,q^2+1}$-minor, then $M$ is GF$(q)$-representable.
\end{conjecture}
The conjecture may hold with $n=3$ for all $q$.
Moreover, the conjecture may also hold when 
``round" is replaced by ``vertically $4$-connected".

\section{Preliminaries}

We assume that the reader is familiar with matroid 
theory; we use the notation and terminology of Oxley~[\ref{oxley}].
A rank-$1$ flat in a matroid is referred to as a {\em point}
and a rank-$2$ flat is a {\em line}. 
A line is {\em long} if it has at least $3$ points.
The number of points in $M$ is denoted $\epsilon(M)$.

Let $M$ be a matroid and let $A,B\subseteq E(M)$.
We define $\sqcap_M(A,B) = r_M(A) + r_M(B)-r_M(A\cup B)$;
this is the {\em local connectivity} between $A$ and $B$.
This definition is motivated by geometry.
Suppose that $M$ is a restriction of PG$(n-1,q)$ and let
$F_A$ and $F_B$ be the flats of PG$(n-1,q)$ that are spanned
by $A$ and $B$ respectively. Then $F_A\cap F_B$ has rank
$\sqcap_M(A,B)$. We say that two sets $A,B\subseteq E(M)$ are {\em skew} if
$\sqcap_M(A,B)=0$.

We let $\cU(l)$ denote the class of matroids with no $U_{2,l+2}$-minor.
Our proof of Theorem~\ref{main} relies heavily on
the following result of Geelen and Kabell~[\ref{gk}, Theorem 2.1].
\begin{theorem}\label{density}
There is an integer-valued function $\alpha(l,q,n)$ such that,
for any positive integers $l,q,n$ with $l\ge q\ge 2$,
if $M\in \cU(l)$ is a matroid with $\epsilon(M)\ge \alpha(l,q,n) q^{r(M)}$,
then $M$ contains a PG$(n-1,q')$-minor for some prime-power $q'>q$.
\end{theorem}

The following result is an important special case of Theorem~\ref{longline}.
\begin{lemma}\label{restriction2}
If $M$ is a round matroid that contains a $U_{2,q+2}$-restriction
and a PG$(2,q)$-restriction, then $M$ has a $U_{2,q^2+1}$-minor.
\end{lemma}

\begin{proof}
Suppose that $M$ is a minimum-rank counterexample.
Let $L,P\subseteq E(M)$ such that
$M|L=U_{2,q+2}$ and $M|P=\PG(2,q)$.
If $M$ has rank $3$, then we may assume that
$E(M) = P \cup \{e\}$.  Since $M|P$ is modular,
$e$ is in at most one long line of $M$. 
Then, since $|P|=q^2+q+1$, we have $\elem(M\con e) \ge q^2+1$ and,
hence, $M$ has a $U_{2,q^2+1}$-minor.
This contradiction implies that $r(M)>3$.
Since $M$ is round, there is an element $e$ that 
is spanned by neither $L$ nor $P$. 
Now $M\con e$ is round and contains both 
$M|L$ and $M|P$ as restrictions. This contradicts 
our choice of $M$.
\end{proof}

The base case of the following lemma is essentially proved in [\ref{gk},~Lemma~2.4]. 
\begin{lemma}\label{skewflat}
Let $\lambda \in \bR$. Let $k$ and $l\ge q \ge 2$ be positive integers,
and let $A$ and $B$ be disjoint
sets of elements in a matroid $M\in \cU(l)$ with $\sqcap_M(A,B)\le k$ and 
$\epsilon_M(A) > \lambda q^{r_M(A)}$. Then there is a set
$A'\subseteq A$ that is skew to $B$ and satisfies
$\epsilon_M(A') > \lambda l^{-k} q^{r_M(A')}$. 
\end{lemma}

\begin{proof}
By possibly contracting some elements in $B-\cl_M(A)$, we may assume that $A$ spans $B$ and thus that $r_M(B) = \sqcap_M(A,B)$. When $k = 1$, this means $B$ has rank $1$. We resolve this base case first. 

Let $e$ be a non-loop element of $B$. We may assume that $A$ is minimal with $\elem_M(A) > \lambda q^{r_M(A)}$, and that $E(M) = A \cup \{e\}$. Let $W$ be a flat of $M$ not containing $e$, such that $r_M(W) = r(M)-2$. Let $H_0, H_1, \dotsc, H_m$ be the hyperplanes of $M$ containing $W$, with $e \in H_0$. The sets $\{H_i - W: 1 \le i \le m\}$ are a disjoint cover of $E(M) - W$. Additionally, the matroid $\si(M \con W)$ is isomorphic to the line $U_{2,m+1}$, so we know that $m \le l$. 

	By the minimality of $A$, we get $\elem_M(H_0 \cap A) \le \lambda q^{r(M)-1}$, so \[\epsilon_M(A-H_0) > \lambda(q-1)q^{r(M)-1}.\] Since the hyperplanes $H_1, \dotsc, H_m$ cover $E(M)-H_0$, a majority argument gives some $1\le i \le m$ such that \[\epsilon_M(H_i \cap A) \ge \frac{1}{m}\epsilon_M(A - H_0) > \frac{\lambda}{l}(q-1)q^{r(M)-1}.\] Setting $A' = A \cap H_i$ gives a set of the required number of points that is skew to $e$ and therefore to $B$, which is what we want. 

Now suppose that the result holds for $k=t$ and consider the case that $k=t+1$. Let $A$ and $B$ be disjoint
sets of elements in a matroid $M$ with $\sqcap_M(A,B)\le t+1$ and  $\epsilon_M(A) > \lambda q^{r_M(A)}$.  As mentioned earlier, we have $r_M(B) = \sqcap_M(A,B) \le t+1$. Let $e$ be any non-loop element of $B$.
By the base case, there exists $A'\subseteq A$ that is skew to $\{e\}$ and satisfies
$\epsilon_M(A') > \lambda l^{-1} q^{r_M(A')}$.
Since $e\not\in \cl_M(A')$ and $r_M(B)\le t+1$, we have
$\sqcap_M(A',B)\le t$. Now the result follows routinely by the induction hypothesis.
\end{proof}

The following two results are used in the reduction
of Theorem~\ref{main} to Theorem~\ref{extra}.
\begin{lemma}\label{jim1}
Let $f(k)$ be an integer-valued function such that
$f(k)\ge  2 f(k-1) - 1$ for each $k\ge 1$ and $f(1)\ge 1$.
If $M$ is a matroid with $\elem(M)\ge f(r(M))$ and $r(M)\ge 1$, then
there is a round restriction $N$ of $M$ such that
$\elem(N)\ge f(r(N))$ and $r(N)\ge 1$.
\end{lemma}

\begin{proof}
We may assume that $M$ is not round and, hence,
there is a partition $(A,B)$ of $E(M)$ such that
$r_M(A)<r(M)$ and $r_M(B)<r(M)$.
Clearly $r_M(A)\ge 1$ and $r_M(B)\ge 1$.
Inductively we may assume that
$\elem_M(A) < f(r_M(A))$ and $\elem_M(B) <f(r_M(B))$.
Thus $\elem(M) \le
\elem(M|A)+\elem(M|B) \le f(r_M(A)) + f(r_M(B))-2 \le 2f(r(M)-1)-2 < f(r(M))$,
which is  a contradiction.
\end{proof}

\begin{lemma}\label{jim2}
Let $q\ge 4$ and $t\ge 1$ be integers and let $M$ 
be a matroid with $\elem(M)\ge \frac{q^{r(M)}-1}{q-1}$
and $r(M)\ge 3t$. If $M$ is not round,
then either $M$ has a $U_{2,q^2+2}$-minor or 
there is a round restriction $N$ of $M$ such that
$r(N)\ge t$ and $\elem(N)>\frac{q^{r(N)}-1}{q-1}$.
\end{lemma}

\begin{proof}
Let $s=r(M)$ and
let $f(k) = \left(\frac q 2\right)^{s-k}\left(\frac{q^{k}-1}{q-1}\right)$.
For any $k\ge 1$,
\begin{align*}
f(k+1) & =  \left(\frac q 2\right)^{s-k-1}\left(\frac{q^{k+1}-1}{q-1}\right) \\
			 & >  \left(\frac q 2\right)^{s-k-1}\left(q \frac{q^{k}-1}{q-1} \right) \\
			 & =  2f(k).
\end{align*}
Moreover $f(1) \ge 1$ and $\elem(M)\ge f(r(M))$.
Then, by Lemma~\ref{jim1},
there is a round restriction $N$ of $M$ 
such that $r(N)\ge 1$ and $\elem(N)\ge f(r(N))$.
Since $M$ is not round, $r(N)<r(M)=s$ and, hence,
$\elem(N) > \frac{q^{r(N)}-1}{q-1}$.
We may assume that $r(N)<t$.
Therefore, since $s\ge 3t$ and $q\ge 4$,
\begin{eqnarray*}
\elem(N) & \ge &
f(r(N)) \\
&=&\left(\frac q 2\right)^{s-r(N)}\left(\frac{q^{r(N)}-1}{q-1}\right)\\
&\ge& \left(\frac q 2\right)^{2t}\left(\frac{q^{r(N)}-1}{q-1}\right)\\
&\ge&q^t\left(\frac{q^{r(N)}-1}{q-1}\right)\\
& > &q^{r(N)}\left(\frac{q^{r(N)}-1}{q-1}\right)\\
&\ge&\left(\frac{q^{r(N)}+1}{q+1}\right)\left(\frac{q^{r(N)}-1}{q-1}\right)\\
&=&\left(\frac{q^{2r(N)}-1}{q^2-1}\right).
\end{eqnarray*}
Therefore, by Theorem~\ref{line},
$M$ has a $U_{2,q^2+2}$-minor, as required.
\end{proof}
	
\section{The main results}

We start with a proof of Theorem~\ref{longline}, which we restate here.
\begin{theorem}\label{longline2}
There is an integer-valued function $n(q)$ such that, for each prime power $q$,
if $M$ is a round matroid that contains a $U_{2,q+2}$-restriction
and a PG$(n(q)-1,q)$-minor, then $M$ has a $U_{2,q^2+1}$-minor.
\end{theorem}

\begin{proof}
Recall that the function $\alpha(l,q,n)$ was defined in Theorem~\ref{density}.
Let $q$ be a prime power, let $\alpha = \alpha(q^2-1,q-1,3)$. 
Let $n$ be an integer that is sufficiently large
so that $\left(\frac{q}{q-1}\right)^n > \alpha q^{5} (q-1)^2$.
We define $n(q)=n$. Suppose that the result fails for this choice of $n(q)$
and let $M$ be a minimum-rank counterexample. Thus $M$ is a round 
matroid having a line $L$, with at least $q+2$ points, and a minor $N$
isomorphic to PG$(n-1,q)$, but $M \in \cU(q^2-1)$.

 Suppose that $N=M\con C\del D$ where $C$ is independent.
 If  $e\in C-L$, then $M\con e$ is round, contains the line $L$,
 and has $N$ as minor --- contrary to our choice of $M$.
 Therefore $C\subseteq L$ and, hence, $r(M)\le r(N)+2\le n+2$.
	 
Let $X=E(M)-L$. By our choice of $n$, we have $\elem(M|(X-D)) \ge  \frac{q^n-1}{q-1}-(q^2+1) 
= q^3 \frac{q^{n-3} -1}{q-1}+q\ge q^{n-1} >  q^4\alpha (q-1)^{n+2}
\ge q^4\alpha (q-1)^{r_M(X)}$.
By Lemma~\ref{skewflat}, there is a flat
$F\subseteq X-D$ of $M$ that is skew to $L$ and satisfies
$\elem(M|F) \ge \alpha (q-1)^{r_M(F)}$.
Since $F$ is skew to $L$, $F$ is also skew to $C$. Therefore $M|F=N|F$
and hence $M|F$ is GF$(q)$-representable. Then,
by Theorem~\ref{density}, $M|F$ has a 
PG$(2,q)$-minor. Therefore there is a set $Y\subseteq F$
such that $(M|F)\con Y$ contains a PG$(2,q)$-restriction.
Now $M\con Y$ is round, contains a $(q+2)$-point line, and
contains a PG$(2,q)$-restriction.
Then, by Lemma~\ref{restriction2}, $M$ has a $U_{2,q^2+1}$-minor.
\end{proof}

Now we will prove Theorem~\ref{extra} which we reformulate
here. The function $n(q)$ was defined in Theorem~\ref{longline2}.
\begin{theorem}\label{extra2}
For each prime power $q$, 
if $M$ is a round matroid with a PG$(n(q)-1,q)$-minor but no 
$U_{2,q^2+1}$-minor, then $\epsilon(M) \le \frac{q^{r(M)}-1}{q-1}$.
\end{theorem}

\begin{proof} 
Let $M$ be a minimum-rank counterexample.
By Lemma~\ref{restriction2}, $r(M)>n(q)$.
Let $e\in E(M)$  be a non-loop element such that $M\con e$ has a
PG$(n-1,q)$-minor. Note that $M\con e$ is  round. Then, by the minimality of $M$,
$\elem(M\con e) \le  \frac{q^{r(M)-1}-1}{q-1}$.
By Theorem~\ref{longline2},
each line of $M$ containing $e$ has at most $q+1$ points.
Hence $\elem(M) \le 1 + q\elem(M\con e)
\le 1+q\left(\frac{q^{r(M)-1}-1}{q-1}\right) = \frac{q^{r(M)}-1}{q-1}$.
This contradiction completes the proof.
\end{proof}

We can now prove our main result, Theorem~\ref{main}, which we restate below. 
\begin{theorem}\label{main2}
Let $l\ge 2$ be a positive integer and let $q$ be the
largest prime power less than or equal to $l$.
If $M$ is a matroid with no $U_{2,l+2}$-minor
and with sufficiently large rank, then 
$\elem(M)\le \frac{q^{r(M)}-1}{q-1}$.
\end{theorem}

\begin{proof}[Proof of Theorem~\ref{main}]
When $l$ is a prime-power, the result follows from Theorem~\ref{line}.
Therefore we may assume that $l\ge 6$ and, hence, $q\ge 5$.
Recall that $n(q)$ is defined in Theorem~\ref{longline2}
and $\alpha(l,q-1,n)$ is defined in Theorem~\ref{density}.
Let $n=n(q)$ and let $k$ be an integer that is sufficiently large
so that $\left(\frac{q}{q-1}\right)^k \ge q\alpha(l,q-1,n)$.
Thus, for any $k' \ge k$, we get $\frac{q^{k'}-1}{q-1} \ge q^{k'-1} \ge \alpha(l,q-1,n) (q-1)^{k'}$.
Let $M\in\cU(l)$ be a matroid of rank at least $3 k$
such that $\elem(M)>\frac{q^{r(M)}-1}{q-1}$.
By Lemma~\ref{jim2}, $M$ has a round restriction $N$
such that we have $r(N)\ge k$ and $\elem(N) > \frac{q^{r(N)}-1}{q-1}\ge  \alpha(l,q-1,n) (q-1)^{r(N)}$.
By Theorem~\ref{density}, $N$ has a PG$(n(q)-1,q')$-minor for some $q' > q-1$. If $q' > q$, then $q' + 1 \ge l+2$, so this projective geometry has a $U_{2,l+2}$-minor, contradicting our hypothesis. We may therefore conclude that $q' = q$, so $N$ has a PG$(n(q)-1,q)$-minor.
Now we get a contradiction by Theorem~\ref{extra2}.
\end{proof}
	
\section{Extremal Matroids}
	
In this section, we prove that the extremal matroids of large rank
for Theorem~\ref{main} are projective geometries. We need the following
result to recognize projective geometries; see Oxley~[\ref{oxley}, Theorem 6.1.1]. 
\begin{lemma}\label{PGaxioms}
Let $M$ be a simple matroid of rank  $n\ge 4$ such that 
every line of $M$ contains at least three points and
each pair of disjoint lines of $M$ is skew.
Then $M$ is isomorphic to $\PG(n-1,q)$ for some prime power $q$. 
\end{lemma}

We can now prove our extremal characterization.
\begin{corollary}\label{extreme}
Let $l\ge 2$ be a positive integer and let $q$ be the
largest prime power less than or equal to $l$.
If $M$ is a simple matroid with no $U_{2,l+2}$-minor,
with $\elem(M)= \frac{q^{r(M)}-1}{q-1}$,
and with sufficiently large rank, then 
$M$ is a projective geometry over GF$(q)$.
\end{corollary}

\begin{proof}
Kung~[\ref{kung}] proved the result for the case that $l$ is a prime-power.
Therefore we may assume that $l\ge 6$ and, hence, $q\ge 5$.
By Theorem~\ref{main}, there is an integer $k_1$
such that, if $M$ is a matroid with no $U_{2,l+2}$-minor
and with $r(M)\ge k_1$, then 
$\elem(M)\le \frac{q^{r(M)}-1}{q-1}$.
Recall that $n(q)$ is defined in Theorem~\ref{longline2} and
 $\alpha(l,q,n)$ is defined in Theorem~\ref{density}.
Let $k_2$ be large enough so that $\left(\frac{q}{q-1}\right)^{k_2} \ge q\alpha(l,q-1,n(q)+2)$, and $k=\max(k_1,k_2)$. 

Let $M\in\cU(l)$ be a simple matroid of rank at least $3 k$
such that $\elem(M)=\frac{q^{r(M)}-1}{q-1}$.
If $M$ is not round, then,
by Lemma~\ref{jim2}, $M$ has a round restriction $N$
such that $r(N)\ge k$ and $\elem(N)> \frac{q^{r(N)}-1}{q-1}$,
contrary to Theorem~\ref{main}. Hence $M$ is round.

From the definition of $k_2$, we get $\epsilon(M) \ge \alpha(l,q-1,n(q)+2)(q-1)^{r(M)}$, so by Theorem~\ref{density}, $M$ has a PG$(n(q)+1,q)$-minor.
Therefore,  by Theorem~\ref{longline2}, each line in $M$ has
at most $q+1$ points. Consider any element $e\in E(M)$.
By Theorem~\ref{main}, $\elem(M\con e)\le \frac{q^{r(M)-1}-1}{q-1}$.
Then
\begin{eqnarray*}
\elem(M) &\le& 1 + q \elem(M\con e)\\
&\le & 1+q\left(\frac{q^{r(M)-1} -1}{q-1}\right)\\
&=& \frac{q^{r(M)}-1}{q-1}\\
&=& \elem(M).
\end{eqnarray*}
The inequalities above must hold with equality. Therefore
each line in $M$ has exactly $q+1$ points.

If $M$ is not a projective geometry, then, by Lemma~\ref{PGaxioms},
there are two disjoint lines $L_1$ and $L_2$ in $M$ such that
$\sqcap_M(L_1,L_2)=1$. Let $e\in L_1$.
Then $L_2$ spans a line with at least  $q+2$ points in $M\con e$.
Since $M$ has a PG$(n(q)+1,q)$-minor, 
$M\con e$ contains a PG$(n(q)-1,q)$-minor;
see~[\ref{ggw}, Lemma 5.2]. 
This contradicts Theorem~\ref{longline2}.
\end{proof}

\section*{Acknowledgements}

We thank the referees for their careful reading of the manuscript and for their useful comments. 

\section*{References}

\newcounter{refs}

\begin{list}{[\arabic{refs}]}%
{\usecounter{refs}\setlength{\leftmargin}{10mm}\setlength{\itemsep}{0mm}}

\item\label{ggw}
J. Geelen, B. Gerards, G. Whittle,
On Rota's conjecture and excluded minors containing
large projective geometries,
J. Combin. Theory Ser. B 96 (2006), 405-425.

\item\label{bk}
J.E. Bonin,  J.P.S. Kung,  The number of points in a combinatorial geometry with no $8$-point-line minors,
Mathematical essays in honor of Gian-Carlo Rota, Cambridge, MA (1996), 271-284, 
Progr. Math., 161, Birkh\"auser Boston, Boston, MA, (1998).

\item\label{gk}
J. Geelen, K. Kabell,
Projective geometries in dense matroids, 
J. Combin. Theory Ser. B 99 (2009), 1-8.

\item\label{gkw}
J. Geelen, J.P.S. Kung, G. Whittle,
Growth rates of minor-closed classes of matroids,
J. Combin. Theory Ser. B 99 (2009), 420-427.

\item\label{kung}
J.P.S. Kung,
Extremal matroid theory, in: Graph Structure Theory (Seattle WA, 1991), 
Contemporary Mathematics, 147, American Mathematical Society, Providence RI, 1993, pp.~21--61.

\item \label{oxley}
J. G. Oxley,  Matroid Theory,
Oxford University Press, New York, 1992.

\end{list}

\end{document}